\documentclass[10pt]{article}

\usepackage{amsmath}
\usepackage{amsthm}
\usepackage{amsfonts}
\usepackage{amssymb}
\usepackage{hyperref}
\usepackage{latexsym}

 \newtheorem{thm}{Theorem}
 \newtheorem{cor}[thm]{Corollary}

 \numberwithin{equation}{section}

\title{On Appell Sets and the Fueter-Sce Mapping}
\author{Norman G{\"u}rlebeck\thanks{Institute of Theoretical Physics, V Hole\v{s}ovick\'ach 2, 180 00 Praha 8 - Hole\v{s}ovice, Czech Republic. {\tt E-mail: norman.guerlebeck@gmail.com}}}

\begin{document}
\maketitle
\date

\begin{abstract}
It is proved, that the recently discussed Appell polynomials in Clifford algebras are the Fueter-Sce extension of the complex monomials $z^k$. Furthermore, it is shown, for which complex functions the Fueter-Sce extension and the extension method using Appell polynomials coincide.   
\end{abstract}
{\bf Keywords}: {Appell sets, Fueter-Sce extension, monogenic polynomials, formal power series}\\[0.2cm]
{\bf AMS-Classification}: 30G35, 33C65

\section{Introduction}

To apply methods of Clifford analysis as advantageously as in complex analysis to partial differential equations the notion 
of $C\ell(n)-$holomorphic functions is indispensable. Therefore, a lot of effort was spent to characterize $C\ell(n)-$holomorphic 
functions and several approaches were developed. But unlike the complex case they do not coincide in general. Among others 
these approaches are a generalized Riemann approach using functions in the kernel of a generalized Cauchy-Riemann system 
\cite{GHS} and a generalized Weierstrass approach using convergent power series expansions \cite{Malonek_1990}. The equivalence of 
these two approaches was shown in \cite{Malonek_1990}. A characterization of $C\ell(n)-$holomorphy as hypercomplex differentiability is also equivalent to the two above given approaches, see \cite{Guerlebeck_Malonek_1999,MitelmanShapiro_1995,Sudbery_1979}. Another method to construct $C\ell(n)-$holomorphic functions in an intrinsic manner, i.e. without using 
functions given already in a lower dimension, uses functional equations \cite{Guerlebeck_2007b}. However, this is only applicable for particular functions. 

Furthermore, it is desirable to extend certain functions given in $\mathbb C$ or $\mathbb R^n$ to $C\ell(n)$-holomorphic functions. These methods are, for instance, the Cauchy-Kowalewskaja extension \cite{BDS} and the Fueter-Sce extension 
\cite{Fueter_1940,Qian_1997,Sce_1957}. In \cite{Sudbery_1979} the extension of real harmonic functions to 
$\mathbb{H}$-holomorphic functions was considered. The construction of conjugate harmonic functions 
was studied in \cite{Brackx_Delanghe_2003}. 

A substantial progress in the extension approach was made in 
\cite{Malonek_Cruz_Falcao_2006}. There, a certain polynomial system, which satisfies a certain differential equation, is used in a formal power series approach. This enables one to construct $C\ell(n)$-holomorphic functions satisfying special differential equations. However, it is not possible to construct all $C\ell(n)$-holomorphic functions with the considered 
polynomial system. Note, that the here given lists of approaches and references do in no way claim to be complete.

In general the intrinsic methods above can also be used as extension methods, if some defining properties are generalized and used to construct a $C\ell(n)$-holomorphic function. However, this will generally only work for certain functions.
Although, starting with the same real analytic function of one real variable or some of their defining properties the mentioned extension and generalization methods will lead in general to different $C\ell(n)-$holomorphic functions. Whereas the extension to complex functions is uniquely determined via the uniqueness theorem of the Taylor expansion of a holomorphic function.

But like shown in \cite{Guerlebeck_2007a} and later used in \cite{Malonek_Falcao_2007} in the special and particular 
important case of the exponential function the Fueter-Sce extension and the power series approach using Appell polynomials 
define the same $C\ell(n)-$holomorphic function. The present work is dedicated to illuminate this connection further and to generalize 
the result obtained for the exponential function. In order to do so the equality of the discussed Appell monomials and the Fueter-Sce extension of the complex polynomials is 
shown for a suitable choice of constants and all functions satisfying the equality between the Fueter-Sce extension and the power series approach using Appell polynomials are determined.

\section{Preliminaries}

Let the vector space $\mathbb R^n$ with an orthonormal basis $\mathbf e_1,\ldots,\mathbf e_n$ be endowed with the product
\begin{align}
	\mathbf e_i\mathbf e_j+\mathbf e_j\mathbf e_i=-2\delta_{ij}.
\end{align}
This generates the Clifford algebra $C\ell(n)$ in which the vector space $\mathbb R^{n+1}$ is embedded identifying $\vec x=(x_0,\ldots,x_n)$ with the paravector $x=x_0+x_1\mathbf e_1+\ldots+ x_n\mathbf e_n=x_0+\mathbf x$.

The power series approach in \cite{Malonek_Cruz_Falcao_2006} is using homogeneous, $C\ell(n)-$holomorphic polynomials $P_n^k$ of the degree $k$ which are introduced as a generalization of the complex monomials $z^k$
\begin{align}
	P^n_k(x)=\sum\limits_{s=0}^k\,^nT^k_s x^{k-s}\overline{x}^{s},\quad P_n^k(1)=1,
\end{align}
where $n$ denotes the dimension of the considered vector space.
The idea of generating holomorphic functions used in \cite{Malonek_Cruz_Falcao_2006} is now, like in the Weierstrass approach in complex analysis, a series expansion with respect to these polynomials. Since only analytic functions defined on the real axis or their complex holomorphic extensions are to generalize here, we restrict ourselves to formal series expansions with real coefficients. The \emph{Appell extension} of a function $f$ with real Taylor coefficients $a_k$ is then given by
\begin{align}\label{Appell_sequence_extension}
	\eta_n[f](x)=\sum\limits_{k=0}^{\infty}a_kP_k^n(x)
\end{align}
in a neighborhood of the origin $x=0$. We remark, that the generalization of complex holomorphic functions with complex coefficients in their series expansion is not anymore possible in a natural and straightforward way.

Let $f:\{z|z\in\mathbb C\wedge\rm{Im}\,z\geq0\}\to\mathbb C$ be a holomorphic function of the form $f(z)=u(w,y)+{\mathrm i} v(w,y)$, where $\mathrm i$ is the imaginary unit and $u,~v:\mathbb R^2\to\mathbb R$. Then the Fueter-Sce transformed function is given by (e.g. \cite{Sproessig_1999})
\begin{align}\label{Fueter_Sce_mapping_0}
	\tau_n[f](x)=\alpha_n[f]\Delta^{\frac{n-1}{2}}\left(u(x_0,|\mathbf x|)+{\boldsymbol \omega}({\mathbf x})v(x_0,|{\mathbf x}|)\right).
\end{align}
The unit vector $\boldsymbol \omega(\mathbf x)$ is defined through $\frac{\mathbf x}{|\mathbf x|}$. The constant $\alpha_n[f]$ is arbitrary and often used to preserve some properties of the function $f$, like a normalization at a certain point. For even dimensions $n$ the operator $\Delta^{\frac{n-1}{2}}$ has to be understood as a Fourier multiplier operator induced by the symbol $(2\pi\mathrm i|\zeta|)^{n-1}$, see \cite{Qian_1997}. In this case the definition \eqref{Fueter_Sce_mapping_0} is not pointwise and hence not considered here. For an odd dimension $n$ \eqref{Fueter_Sce_mapping_0} simplifies to (see \cite{Sproessig_1999})

\begin{align}\label{Fueter_Sce_mapping}
	\tau_n[f](x)=\alpha_n[f] \left(\left(\frac 1 {|{\mathbf x}|} \partial_{|{\mathbf x}|}\right)^{\frac{n-1} 2}u(x_0,|{\mathbf x}|)+{\boldsymbol \omega}({\mathbf x})\left( \partial_{|{\mathbf x}|}\frac 1 {|{\mathbf x}|}\right)^{\frac{n-1} 2}v(x_0,|{\mathbf x}|)\right).
\end{align}

\section{The Fueter-Sce mapping of the \texorpdfstring{complex\\ mo\-no\-mi\-als}{}}

Now we study the Fueter-Sce map of the complex monomials. Our particular interest lies in the choice of the normalization constant.
First, it is necessary to split $z^k$ in the real and the imaginary part, where $z=w+{\mathrm i} y$ with $w,~y\in{\mathbb R}$:
\begin{align}
	z^k=\sum\limits_{l=0}^{k}\binom{k}{l}({\mathrm i}y)^lw^{k-l}=u_k(w,y)+{\mathrm i}v_k(w,y).
\end{align}
The functions $u_k,~v_k:{\mathbb R}^2\to{\mathbb R}$ are defined by
\begin{align}\label{definition_u_k_v_k}
	u_k(w,y)&=\sum\limits_{p=0}^{\left\lfloor \frac k 2 \right\rfloor}\binom{k}{2p}(-1)^py^{2p}w^{k-2p}\\
	v_k(w,y)&=\sum\limits_{p=0}^{\left\lfloor \frac {k-1} 2 \right\rfloor}\binom{k}{2p+1}(-1)^py^{2p+1}w^{k-2p-1}.
\end{align}
In these formulas $\lfloor q\rfloor$ stands for the floor function defined by $\max(k\in\mathbb Z\,|\, k\leq q)$.
Applying the operators of equation \eqref{Fueter_Sce_mapping} we get

\begin{align}\label{apply_operators}
\begin{split}
	\left(\frac 1 {|{\mathbf x}|} \partial_{|{\mathbf x}|}\right)^{\frac{n-1} 2}u_k(x_0,|{\mathbf x}|)
	=&\sum\limits_{p=0}^{\left\lfloor \frac k 2 \right\rfloor}\binom{k}{2p}(-1)^px_0^{k-2p}
	\beta_{n,2p}\left(|{\mathbf x}|\right)\\
	\left(\partial_{|{\mathbf x}|}\frac 1 {|{\mathbf x}|}\right)^{\frac{n-1} 2}v_k(x_0,|{\mathbf x}|)
	=&\sum\limits_{p=0}^{\left\lfloor \frac {k-1} 2 \right\rfloor}\binom{k}{2p+1}(-1)^px_0^{k-2p-1}\beta_{n,2p+1}\left(|{\mathbf x}|\right),
\end{split}
\end{align}
where
\begin{align}\label{definition_beta}
	\beta_{n,2p+1}\left(|{\mathbf x}|\right)=|{\mathbf x}|\beta_{n,2p}\left(|{\mathbf x}|\right)=	
	\begin{cases}
		\frac{(2p)!!}{(2p-n+1)!!}|{\mathbf x}|^{2p-n+2}&, \text{if}\quad 2p\geq n-1\\
		0&, \text{if}\quad 2p<n-1.
	\end{cases}
\end{align}
The equations \eqref{apply_operators} and \eqref{definition_beta} imply
\begin{align}\label{first_n_monomials}
	\tau_n\left[z^k\right]=0\quad\forall k<n-1.
\end{align}
Therefore, in the following only values $k\geq n-1$ are considered. The constants $\alpha_n\left[z^k\right]$ are fixed by the demand $\tau_n\left[z^k\right](1)=1$, which clearly generalizes the corresponding property of $z^k$ from the complex case. For this purpose the functions $\beta_{j,n}$ from \eqref{definition_beta}  are evaluated at $\mathbf x=0$:
\begin{align}\label{beta_at_zero}
	\beta_{2p,n}(0)=\delta_{2p,n-1}(n-1)!!,\quad\beta_{2p+1,n}(0)=0.
\end{align}
Here $\delta_{i,j}$ denotes the Kronecker symbol and $0!!=1$ is taken into account. With \eqref{Fueter_Sce_mapping}, \eqref{apply_operators} and \eqref{beta_at_zero} we obtain for the constants
\begin{align}\label{alpha_fixed}
	\alpha_n\left[z^k\right]=(-1)^{\frac{n-1}{2}}\frac{(n-2)!!(k-n+1)!}{k!}\quad\forall k\geq n-1.
\end{align}

\section{Comparison with the Appell sequence \texorpdfstring{$P_k^n$}{}}

Appell sequences are defined as polynomial sequences $(S_k)_{k\in\mathbb{N}}$ satisfying the differential equation 
\begin{align}
	{\frac{\rm d}{{\rm d}x}} S_k(x) = kS_{k-1}(x)
\end{align}
and ${\rm deg}\,S_k=k$. Originally introduced for polynomials of one complex variable \cite{Appell_1880}, this definition is naturally extensible to Clifford valued functions of one paravector. Other definitions using a formal power series representation or functional equations are possible, see \cite{Chihara_1978}. 
The here considered Appell sequence $(P_k^n)_{k\in\mathbb N}$ with
\begin{align}\label{Appell_polynomials}
		P^n_k(x)=\sum\limits_{s=0}^k\,^nT^k_s x^{k-s}\overline{x}^{s},\quad P^n_k(1)=1
\end{align}
was introduced in \cite{Malonek_Cruz_Falcao_2006}. An explicit formula for ${}^nT_s^k$ was given in two dimensions in \cite{Malonek_Cruz_Falcao_2006} and for arbitrary dimensions in \cite{Malonek_Falcao_2007}. We are going to use only the alternating sums of the constants ${}^nT_s^k$ also given in \cite{Malonek_Falcao_2007} by
\begin{align}\label{definition_c_n_k}
	c^k_n=\sum\limits_{s=0}^{k}{}^nT_s^k(-1)^s=
	\displaystyle\begin{cases}
		\displaystyle\frac{(k-1)!!(n-2)!!}{(n+k-2)!!}, & \text{if $k$ is even}\\
		\displaystyle\frac{k!!(n-2)!!}{(n+k-1)!!},& \text{if $k$ is odd}.
	\end{cases}
\end{align}

We will now prove, that this particular Appell sequence can also be obtained using a Fueter-Sce transformation of the complex polynomials $z^k$.
\begin{thm}\label{Fueter_Sce_equals_Appell}
	For all odd dimensions $n>1$ and all paravectors in $C\ell(n)$ holds
	\begin{align}
		\tau_n\left[z^{k+n-1}\right](x)=P_k^n(x),
	\end{align}
	if the constant in the Fueter-Sce transformation is chosen according to \eqref{alpha_fixed}.
\end{thm}
\begin{proof}
To prove this theorem it is sufficient to restrict the arguments $x$ to vectors $\mathbf x$. If we have shown the equality in this case, arbitrary arguments are considered by using the uniqueness of the Cauchy-Kowalewskaja extension. For vectors $\mathbf x$ the polynomials of the Appell sequence \eqref{Appell_polynomials} are of the form
\begin{align}
	P^n_k(\mathbf x)=\sum\limits_{s=0}^k{}^nT_s^k(-1)^s\mathbf x^k=c_n^k\mathbf x^k.
\end{align}

The polynomials $\tau_n\left[z^l\right]$ for vectors $\mathbf x$ are obtained from \eqref{Fueter_Sce_mapping}, \eqref{apply_operators}, \eqref{definition_beta} and \eqref{alpha_fixed} for $x_0=0$ and $l\geq n-1$:
\begin{align}\label{tau_n_k_calculated}
\begin{split}
	\tau_n\left[z^l\right](\mathbf x)=&\alpha_n\left[z^l\right]
	\begin{cases}
		\displaystyle(-1)^{\frac l 2}\beta_{n,l}(|\mathbf x|), & \text{if $l$ is even}\\
		\displaystyle(-1)^{\frac {l-1} 2}\boldsymbol{\omega}(\mathbf x)\beta_{n,l}(|\mathbf x|), & \text{if $l$ is odd}\\
	\end{cases}\\
	=&
	\begin{cases}
		\displaystyle\frac{(l-n)!!(n-2)!!}{(l-1)!!}\mathbf x^{l-n+1}, & \text{if $l$ is even}\\
		\displaystyle\frac{(l-n+1)!!(n-2)!!}{l!!}\mathbf x^{l-n+1}, & \text{if $l$ is even}.
	\end{cases}
\end{split}
\end{align}
Writing $k+n-1$ for $l$ with $k\geq 0$
\begin{align}
\tau_n\left[z^{k+n-1}\right](\mathbf x)=P_k^n(\mathbf x)
\end{align}
can be concluded. Therefore, both functions coincide in the hyperplane $\mathbb R^n$. Furthermore, both functions are $C\ell(n)-$holomorphic due to the constructions. The problem of expanding a sufficiently smooth function defined in $\mathbb R^n$ to a $C\ell(n)$-holomorphic function in $\mathbb R^{n+1}$ has at least locally a unique solution - the Cauchy-Kowalewskaja extension. Thus, the functions $P^n_k$ and $\tau_n\left[z^{k+n-1}\right]$ are equal in a neighborhood of $\mathbb R^{n}$ in $\mathbb R^{n+1}$. However, the domains of both functions coincide. Hence, they are equal in $\mathbb R^{n+1}$.
\end{proof}

At a first glance one might think, that the result of the theorem \ref{Fueter_Sce_equals_Appell}, the equality of the Fueter-Sce extension of the monomials $z^k$ and the Appell polynomials, is extensible to all analytic functions of a real variable. But this is not true, because different degrees $k$ require different choices of the constants $\alpha_n[z^k]$ in order to preserve the normalization. This problem is associated with the property of the Fueter-Sce extension to be only linear in two functions $f,~g$ with a non-vanishing Fueter-Sce extension, if and only if $\alpha_n[f+g]=\alpha_n[f]=\alpha_n[g]$ holds. Nevertheless, theorem \ref{Fueter_Sce_equals_Appell} can be generalized to a special class of analytic functions defined on the real line. This is done in the next two sections.

\section{Comparison of both extension methods}

The theorem \ref{Fueter_Sce_equals_Appell} implies the following:

\begin{thm}\label{theorem_recursion_relation}
Let $f$ be a holomorphic function in $\mathbb C$, which has real Taylor coefficients $a_k$ in the neighborhood of $z=0$. Then for any odd dimension $n>1$ the equality
\begin{align}
	\tau_n[f](x)=\eta_n[f](x)
\end{align}
holds for all paravectors $x\in Cl(n)$, if and only if the recurrence formula
	\begin{align}\label{recursion_relation}
		a_{k+n-1}=\frac{\gamma_n[f]k!}{(k+n-1)!}a_{k}
	\end{align}
	is satisfied. The constant $\gamma_n[f]$ depends on the dimension $n$ and the function $f$, but not on $k$.
\end{thm}
\begin{proof}
	To prove this result the Fueter-Sce mapping \eqref{Fueter_Sce_mapping} is applied to
	\begin{align}
		f(z)=\sum\limits_{k=0}^{\infty}a_kz^k=\sum\limits_{k=0}^{\infty}a_ku_k(w,y)+
		\mathrm i\sum\limits_{k=0}^{\infty}a_kv_k(w,y),
	\end{align}
	where $z=w+\mathrm iy$ with $w,~y\in\mathbb R$, $a_k\in\mathbb R$ and $u_k,~v_k$ defined like in
	\eqref{definition_u_k_v_k}. Because $f$ is holomorphic the Fueter-Sce extension gives
	\begin{align}
		\tau_n[f](x)\hspace{-0.05cm}=\hspace{-0.05cm}\alpha_n[f]\sum\limits_{k=0}^{\infty}\hspace{-0.05cm}a_k\hspace{-0.05cm}\left(\left(\frac 1 {|{\mathbf x}|}
		\partial_{|{\mathbf x}|}\right)^{\frac{n-1} 2}\hspace{-0.3cm}u_k(x_0,|{\mathbf x}|)\hspace{-0.05cm}+\hspace{-0.05cm}{\boldsymbol \omega}({\mathbf x})\hspace{-0.05cm}
		\left(\partial_{|{\mathbf x}|}\frac {1}{|{\mathbf x}|}\right)^{\frac{n-1} 2}\hspace{-0.3cm}v_k(x_0,|{\mathbf x}|)
		\right).
	\end{align}
	With equation \eqref{first_n_monomials} and theorem \ref{Fueter_Sce_equals_Appell} this can be
  rewritten in the form
	\begin{align}\label{arbitrary_fuction_FS}
	\begin{split}
		\tau_n[f](x)&=\alpha_n[f]\sum\limits_{k=n-1}^{\infty}\frac{a_k}{\alpha_n\left[z^k\right]}
		\tau_n\left[z^k\right](x)=\alpha_n[f]\sum\limits_{k=0}^{\infty}\frac{a_{k+n-1}}
		{\alpha_n\left[z^{k+n-1}\right]}P^n_k(x).
	\end{split}
	\end{align}
	The polynomials $P^n_k$ are linearly independent for different degrees $k$. Hence, the series expansion
  in \eqref{arbitrary_fuction_FS} equals the Appell extension of $f$ if and only if the recurrence formula \eqref{recursion_relation} is 
  satisfied due to \eqref{alpha_fixed}, where $\gamma_n[f]$ is given by
  \begin{align}
  	\frac{(-1)^{\frac{n-1}{2}}(n-2)!!}{\alpha_n[f]}.
  \end{align}\qedhere
  \end{proof}
In the following a closed representation formula for all functions satisfying the recurrence formula \eqref{recursion_relation} is obtained. Firstly, the recurrence formula is solved and afterwards all corresponding holomorphic functions are determined.

\begin{cor}
	For any odd dimension $n>1$ the solution of the recurrence formula \eqref{recursion_relation} is given by
	\begin{align}\label{explicit_solution_recursion_relation}
		a_m=\left(\gamma_n[f]\right)^l\frac{r!}{((l+1)(n-1)+r)!}a_r,	
	\end{align}
	such that $m=l(n-1)+r$, $r\in\{0,1,..,n-2\}$ and $l\in\mathbb N$. The real coefficients $a_0,\ldots,a_{n-2}$
  are the initial conditions and can be chosen arbitrarily.
\end{cor}
\begin{proof}
	The proof is explicit by using telescoping products. If $l>1$ and the abbreviation $\xi_l=(l(n-1)+r)!$ is used, then the equations
	\begin{align}
	\begin{split}
		a_{l(n-1)+r}&=\gamma_n[f]\frac{\xi_l}{\xi_{l+1}}a_{(l-1)(n-1)+r}=\ldots=\left(\gamma_n[f]\right)^l \frac{\xi_l}{\xi_{l+1}}\frac{\xi_{l-1}}{\xi_{l}}\ldots\frac{\xi_1}{\xi_2}\frac{\xi_0}{\xi_1}a_r
	\end{split}
	\end{align}
	are obtained which leads directly to \eqref{explicit_solution_recursion_relation}.
\end{proof}
Endowed with the explicit solution of the recurrence formula \eqref{explicit_solution_recursion_relation}
we can state theorem \ref{theorem_recursion_relation} more precisely.
\begin{thm}\label{fct_class}
Let $f$ be a holomorphic function in $\mathbb C$, which has real Taylor coefficients $a_k$ in the neighborhood of $z=0$. Then in any odd dimension $n>1$ the equality
\begin{align}
	\tau_n[f](x)=\eta_n[f](x)
\end{align}
holds, if and only if $f(z)$ is of the form
\begin{align}\label{closed_representation}
  \sum\limits_{r=0}^{n-2}\frac{r!a_rz^r}{(n-1+r)!}\,_1F_{n-1}\left(1;\frac{n+r}{n-1},
	\frac{n+r+1}{n-1},\ldots,\frac{2n+r-2}{n-1};
	\frac{\gamma_n[f]z^{n-1}}{(n-1)^{n-1}}\right),
\end{align}
where ${}_1F_{n-1}$ is a generalized hypergeometric function and $\gamma_n[f]$ a arbitrary real constants. 
\end{thm}
\begin{proof}
	With the explicit solution of the recurrence formula we can rewrite the function $f$ and due to the holomorphy rearrange the sum
	\begin{align}\label{series_expansion_with_recursion_relation}
	\begin{split}
		f(z)&=\sum\limits_{r=0}^{n-2}\sum\limits_{l=0}^{\infty}a_{l(n-1)+r}z^{l(n-1)+r}\\
		&=\sum\limits_{r=0}^{n-2}	r!a_rz^r\sum\limits_{l=0}^{\infty}
		\frac{\left(\gamma_n[f]\right)^l}{((l+1)(n-1)+r)!}z^{l(n-1)}.
	\end{split}
	\end{align}
	Using the definition of the generalized hypergeometric function we obtain
	\begin{align}\label{hypergepmetric_function_definition}
	\begin{split}
		_1&F_{n-1}\left(1;\frac{n+r}{n-1},\frac{n+r+1}{n-1},\ldots,\frac{2n+r-2}{n-1};
		\frac{\gamma_n[f]z^{n-1}}{(n-1)^{n-1}}\right)\\
		=&\sum\limits_{l=0}^{\infty}{\gamma_n[f]^l\left(\frac{z}{n-1}\right)^{l(n-1)}}\Phi_l,
	\end{split}
	\end{align}
	where $\Phi_l$ is defined by
	\begin{align}
	\begin{split}
		\Phi_l=\prod\limits_{s=0}^{n-2}\frac{\Gamma(\frac{n+s+r}{n-1})}{\Gamma(\frac{n+s+r}{n-1}+l)}.
	\end{split}
	\end{align}
	Simple manipulations of this formula lead to
	\begin{align}
		\Phi_l^{-1}=\prod\limits_{s=0}^{n-2}\prod\limits_{j=0}^{l-1}\left(\frac{n+s+r}{n-1}+j\right)
		=\frac{1}{(n-1)^{l(n-1)}}\prod\limits_{j=0}^{l-1}\frac{((j+2)(n-1)+r)!}{((j+1)(n-1)+r)!}.
	\end{align}
	Again a telescoping product simplifies the calculations and we get
	\begin{align}\label{phi_calculated}
		\Phi_l=(n-1)^{l(n-1)}\frac{(n-1+r)!}{((l+1)(n-1)+r)!}. 
	\end{align}
	Inserting \eqref{phi_calculated} in \eqref{hypergepmetric_function_definition} and the result in
	\eqref{closed_representation}	the form \eqref{series_expansion_with_recursion_relation} is obtained,
	 which completes the proof.
\end{proof}
If we ask whether a function $f$ is of the form \eqref{closed_representation} it can be easier to show that the recurrence formula is satisfied. As an example the results concerning a hypercomplex exponential function in \cite{Guerlebeck_2007a,Malonek_Falcao_2007} are rediscovered in a very simple manner. The coefficients of the Taylor series of the real exponential function are given by $a_m=\frac 1 {m!}$. 
Inserting this in the recurrence formula \eqref{recursion_relation} leads to
\begin{align}
	\gamma_n\left[{\mathrm e}^z\right]=1,
\end{align}
which does not depend on $k$. Therefore, the recurrence formula is satisfied and Fueter-Sce extension and the Appell extension of the exponential function coincide. Thus, the result of \cite{Guerlebeck_2007a,Malonek_Falcao_2007} is included as a special case if the results about the Fueter-Sce extension of the exponential function in \cite{Sproessig_1999} are taken into account. 
Furthermore, we obtained decompositions of the complex exponential function $\mathrm e^z$ with respect to certain generalized hypergeometric functions
\begin{align}
	\mathrm e^z\hspace{-0.1cm}=\sum\limits_{r=0}^{n-2}\frac{z^r}{(n-1+r)!}\,_1F_{n-1}\left(1;\frac{n+r}{n-1},
	\frac{n+r+1}{n-1},\ldots,\frac{2n+r-2}{n-1};\frac{z^{n-1}}{(n-1)^{n-1}}\right),\notag
\end{align}
where the parameter $n$ is odd and $n\geq3$.

In the same manner, it can be easily shown, that the hyperbolic sine and hyperbolic cosine are in the class of functions described in theorem \ref{fct_class}.

\section*{Acknowledgment}

The first steps toward these results were done during a SOCRATES financed stay at the University Aveiro. The author wants to thank H. Malonek for a lot of discussions and his great interest in this work. Many useful remarks of \mbox{K. G\"urlebeck} during the preparation of the manuscript are gratefully acknowledged.
\end{document}